\documentclass[12pt,a4paper]{article}

\usepackage{amssymb}
\usepackage{amsthm}
\usepackage{amsmath}
\usepackage[colorlinks=true,allcolors=black]{hyperref}
\usepackage{graphics}
\usepackage{enumerate}
\usepackage{tikz}
\newtheorem{theorem}{Theorem}[section]

\newtheorem{lemma}[theorem]{Lemma}

\newtheorem{remark}[theorem]{Remark}

\begin{document}

	\title{Characterizations of $p$-groups whose power graphs satisfy certain connectivity conditions}
	\author{Ramesh Prasad Panda\thanks{Department of Mathematics, School of Advanced Sciences, VIT-AP University, Amaravati, PIN-522237, Andhra Pradesh, India.} \thanks{Email address: rameshprpanda@gmail.com} }
	
	\date{}

	\maketitle
	
	\begin{abstract}
		Let $\Gamma$ be an undirected and simple graph. A set $ S $ of vertices in $\Gamma$ is called a {cyclic vertex cutset} of $\Gamma$ if $\Gamma - S$ is disconnected and has at least two components containing cycles. If $\Gamma$ has a cyclic vertex cutset, then it is said to be {cyclically separable}. The {cyclic vertex connectivity} of $\Gamma$ is the minimum of cardinalities of the cyclic vertex cutsets of $\Gamma$. The {power graph} $\mathcal{P}(G)$ of a group $G$ is the undirected and simple graph whose vertices are the elements $G$ and two vertices are adjacent if one of them is the power of other in $G$. In this paper, we first characterize the finite $ p $-groups ($p$ is a prime number) whose power graphs are cyclically separable in terms of their maximal cyclic subgroups. Then we characterize the finite $ p $-groups whose power graphs have equal vertex connectivity and cyclic vertex connectivity.

		\vskip .5cm
		
		

	\end{abstract}

	\section{Introduction}
	
	All graphs considered in this paper are undirected and simple. Let $\Gamma$ be a graph. The {\it vertex connectivity} of $\Gamma$, denoted by $\kappa(\Gamma)$, is the minimum number of vertices whose deletion results in a disconnected or trivial subgraph of $\Gamma$. A \emph{vertex cutset} of $\Gamma$ is a set $ S $ of vertices in $\Gamma$ such that $\Gamma - S$ is disconnected. If $\Gamma$ is not a complete graph, then the vertex connectivity of $\Gamma$ is the minimum of cardinalities of the cutsets of $\Gamma$. 
	A \emph{cyclic vertex cutset} of $\Gamma$ is a vertex cutset $S$ of $\Gamma$ such that $\Gamma - S$ has at least two components containing cycles. If $\Gamma$ has a cyclic vertex cutset, then it is said to be {\em cyclically separable}. The \emph{cyclic vertex connectivity} $ c\kappa(\Gamma) $ is the minimum of cardinalities of the cyclic vertex cutsets of $\Gamma$.  If $\Gamma$ has no cyclic vertex cutset, $ c\kappa(\Gamma) $ is taken as infinity. The \emph{cyclic edge connectivity} is defined analogously by replacing vertex deletion with edge deletion. The notion of cyclic  connectivity of a graph first appeared in the famous incorrect conjecture of Tait in 1880, which was an attempt to prove the four color conjecture \cite{tait1880}. Birkhoff \cite{birkhoff1913} later reduced the four color conjecture from all planar graphs to a class of planar cubic graphs by making use of cyclic  connectivity. Other applications of this graph parameter include problems such as integer flow conjectures \cite{zhang1997} and measures of network reliability \cite{latifi1994}. Cyclic  connectivity of a graph has been studied in many other contexts, see \cite{liu2011,liu2022,nedela2022,robertson1984} and the references therein.
	
	The \emph{power graph} $\mathcal{P}(G)$ of a group $G$ is the graph whose vertices are the elements $G$ and two vertices are adjacent if one of them is the power of other in $G$. These graphs were introduced for groups and semigroups by Kelarev and Quinn \cite{kel-2000,kel-2002}. Power graphs are one of numerous other graphs defined on groups that have been studied extensively in literature (e.g., Cayley graphs \cite{Klotz}, commuting graphs \cite{bertram},  conjugacy class graphs \cite{Chen}, and enhanced power graphs \cite{aalipour,parveen}).
	In recent years, power graphs have been actively studied by researchers, resulting in many interesting results (see \cite{AKC,cur-2016,PPS-cyclic,rather,zahirovic2022}). In \cite{chattopadhyay2021,chattopadhyay2019,panda2018a}, vertex connectivity of power graphs of finite groups were studied. It was observed that in \cite{panda2018b} that the edge connectivity of the power graph of a finite group is equal to its minimum degree. Subsequently, in \cite{PPS-cyclic,panda2023}, the minimum degree of power graphs of finite groups were studied. In \cite{moghaddamfar2014,panda2018a}, some connectedness properties of power graphs of finite $p$-groups were investigated.
	
	With the introduction of power graphs, relationships between them and the underlying groups have been studied from different viewpoints \cite{AKC}. Moghaddamfar et al. \cite{moghaddamfar2014} characterized all finite groups whose power graphs are a bipartite graph, a planar graph, or a  strongly regular graph. Whereas, Ma and Feng \cite{xuanlong2015} classified all finite groups
	whose power graphs are uniquely colorable, split or unicyclic. 
	In \cite{panda2020}, Panda characterized the finite groups with minimally connected power graphs. Moreover, in \cite{panda2024}, Panda et al. gave a characterization of the finite nilpotent groups whose power graphs have equal  vertex connectivity and minimum degree. In this paper, we address two characterization problems associated with connectivies of power graphs of finite $p$-groups, where $p$ is a prime number. We first characterize the finite $ p $-groups whose power graphs are cyclically separable in terms of their maximal cyclic subgroups. Then we characterize the finite $ p $-groups whose power graphs have equal  vertex connectivity and cyclic  vertex connectivity.

	For any set $A$, let $|A|$ be its number of elements. Let $G$ be a group. Two subgroups $H_1$ and $H_2$ of $G$ are said to have \emph{trivial intersection} if $H_1 \cap H_2 = \{e\}$, where $e$ is the identity element of $G$. A cyclic subgroup of $G$ is called a {\it maximal cyclic subgroup} of $G$ if it is not properly contained in any other cyclic subgroup of $G$. Let $\mathcal{M}(G)$ be the set of all maximal cyclic subgroups of $G$. 
	For $M, N \in \mathcal{M}(G)$, we write $\text{d}(M,N) = \min\{|M {\setminus} N|, |N {\setminus} M|\} $. 
	The \emph{difference number} of $G$ is denoted by $\partial(G)$ and is defined as the maximum of $\text{d}(M,N)$ as $M$ and $N$ vary over every pair of subgroups in $\mathcal{M}(G)$. 
	For example, consider the elementary abelian group $\mathbb{Z}_2 \times \mathbb{Z}_2 \times \mathbb{Z}_2$. It has seven non-identity elements, namely $(0,0,1), (0,1,0), (0,1,1), (1,0,0), (1,0,1), (1,1,0)$, and $(1,1,1)$. Each of these elements generates a maximal cyclic subgroup of order $2$. Hence $\partial(\mathbb{Z}_2 \times \mathbb{Z}_2 \times \mathbb{Z}_2) = 1$. Whereas, the quaternion group $Q_8 = \langle a, b \mid a^4 = e, a^2 = b^2, ab = ba^{-1} \rangle$ has exactly three maximal cyclic subgroups; they are $\langle a \rangle$, $\langle ab \rangle$ and $\langle a^2b \rangle$. Since each of them has order $4$ and the intersection of every pair of them is $\{e, a^2\}$, we get $\partial(Q_{8}) = 2$.
	
	Throughout the paper, $p$ is a prime number. Let $G$ be a group.  We observe that $\partial(G) \geq 1$ if and only if $G$ is non-cyclic. Whereas, $\partial(G) \geq 2$ if and only if $G$ has no maximal subgroup of order $2$. We prove in the following theorem that for any finite \break $p$-group $G$, the conditions that $\mathcal{P}(G)$ is cyclically separable and the inequality $\partial(G) \geq 3$ are equivalent. In fact, we obtain a characterization of these conditions in terms of the values of $p$ and the maximal cyclic subgroups of $G$.
	
	\newpage
	
	\begin{theorem}\label{mainthm}
		For any finite $p$-group $G$, the following statements are equivalent:
		\begin{enumerate}[\rm(i)]
			\item $\mathcal{P}(G)$ is cyclically separable,
			\item $G$ satisfies the inequality $\partial(G) \geq 3$,
			\item $G$ satisfies one of the following conditions:
			\begin{enumerate}[\rm(a)]
				\item $p>3$ and $G$ is non-cyclic,
				\item $p=3$ and $G$ has at least two maximal cyclic subgroups of order greater than $3$,
				\item $p=2$ and $G$ has at least two maximal cyclic subgroups of order greater than $4$, or that $G$ has at least two maximal cyclic subgroups of order greater than $2$ with trivial intersection.
			\end{enumerate}
		\end{enumerate}
	\end{theorem}
	
	For any graph $\Gamma$, we have $ \kappa(\Gamma) \leq c\kappa(\Gamma)$. Thus the question arises about the equality of these two connectivity parameters. In the following theorem, we address this for the power graph of a finite $p$-group.
	
	\begin{theorem}\label{mainthm2}
		Let $G$ be a finite $p$-group. Then $ \kappa(\mathcal{P}(G)) = c\kappa(\mathcal{P}(G)) $ if and only if $G$ satisfies one of the following conditions:
		\begin{enumerate}[\rm(i)]
			\item $p>3$ and $G$ is non-cyclic,
			\item $p \in \{2,3\}$ and $G$ has at least two maximal cyclic subgroups of order greater than $p$ with trivial intersection.
		\end{enumerate}
	\end{theorem}
	
	The \emph{enhanced power graph} of a group $G$ is the undirected and simple graph whose vertices are the elements $G$ and two vertices are adjacent if they belong to the same cyclic subgroup of $G$. Enhanced power graphs were introduced in \cite{aalipour} and have been studied actively since then (see the survey \cite{ma2022survey} and the reference therein).  The power graph of $G$ is a spanning subgraph of its enhanced power graph. It was shown in \cite{aalipour} that for any finite group $G$, its power graph and enhanced power graph are equal if and only if  every cyclic subgroup of $ G $ has prime power order. In view of this and the above theorems, we state the following remark.
	
	\begin{remark}
		For any finite $p$-group $G$, Theorems \ref{mainthm} and \ref{mainthm2} hold if the power graph of $G$ is replaced by the enhanced power graph of $G$.
	\end{remark}

	We conclude the introduction by fixing some necessary notations for  groups and graphs.
	We always denote by $e$ the identity element of the group under consideration.
	Let $G$ be a group. For any $A \subseteq G$, we write $A^* = A {\setminus} \{e\}$. We use the notation $\mathcal{P}(A)$ for the subgraph of $ \mathcal{P}(G) $ induced by $A$. Also, we simply write $\mathcal{P}(e)$ instead of $\mathcal{P}(\{e\})$. For any cyclic subgroup $\langle x \rangle$ generated by $x \in G$, we write $[x] = \{y \in G : \langle y \rangle = \langle x \rangle \}$. If $G$ is a cyclic group, we denote by $\widetilde{G}$, the set of generators of $G$. That is, if $G = \langle x \rangle$, then $\widetilde{G} = [x]$.

	If $\Gamma_1, \Gamma_2, \dots \Gamma_r$ are pairwise disjoint graphs, we refer to their union as a \emph{disjoint union}, and denote it by $\Gamma_1 + \Gamma_2 + \dots + \Gamma_r$. Whereas, if $\Gamma_1 = \Gamma_2 = \dots = \Gamma_r = \Gamma$, we denote the above disjoint union by $n\Gamma$. Let $\Gamma_1$ and $\Gamma_2$ be two graphs.  We write $\Gamma_1 \cong \Gamma_2$ if the two graphs are isomorphic.
	The \emph{join} $\Gamma_1 \vee \Gamma_2$ of $\Gamma_1$ and $\Gamma_2$ is the graph obtained by taking $\Gamma_1 + \Gamma_2$ and then adding edges $\{v_1,v_2\}$ for all vertices $v_1$ and $v_2$ in $\Gamma_1 $ and $ \Gamma_2$, respectively.
	
	\section{Graph connectivities}
	
	For any finite group $G$, the identity vertex is adjacent to all other vertices of $\mathcal{P}(G)$. So $\mathcal{P}(G)$ is a connected graph. The following lemma talks about the completeness of $\mathcal{P}(G)$. 
	
	\begin{lemma}[\cite{CGS-2009}] \label{lemma}
		For any finite group $G$, $\mathcal{P}(G)$ is a complete graph if and only if $G$ is a cyclic group of prime power order.
	\end{lemma}
	
	The next lemma is about the adjacency of vertices belonging to different subgroups in a power graph.
	
	\begin{lemma}
		\label{lem_adj}
		For any group $G$ with subgroups $H$ and $K$, if $x \in H {\setminus} K$ and $y \in K {\setminus} H$, then $x$ and $y$ are not adjacent in $\mathcal{P}(G)$. 
	\end{lemma}
	
	In particular, in Lemma \ref{lem_adj}, if the intersection of $H$ and $K$ is trivial, and $x \in H^*$ and $y \in K^*$, then $x$ and $y$ are not adjacent in $\mathcal{P}(G)$. We will frequently use the above lemmas without referring to them explicitly.
	
	We first prove that for any finite group $G$, the cyclic separability of $\mathcal{P}(G)$ is necessary for $\partial(G) \geq 3$ to hold. 
	
	\begin{lemma}\label{lemma1}
		For any finite group $G$, if $\partial(G) \geq 3$ holds, then $\mathcal{P}(G)$ is cyclically separable.
	\end{lemma}
	
	\begin{proof}
		As $\partial(G) \geq 3$, $G$ has two maximal cyclic subgroups $M$ and $N$ such that $|M{\setminus} N| \geq 3$ and $|N{\setminus} M| \geq 3$. We notice that the subgraph of $\mathcal{P}(G)$ induced by $M \bigtriangleup N = (M{\setminus} N) \cup (N{\setminus} M)$ is a disconnected  graph.
		
		If $M$ is of prime power order, then $M {\setminus} N$ is a clique in $\mathcal{P}(G)$. Additionally, as $|M{\setminus} N| \geq 3$, the subgraph of $\mathcal{P}(G)$ induced by $M {\setminus} N$ contains a cycle. Next, suppose that $M$ is not of prime power order. Thus $pq \mid |M|$ for some distinct primes $p$ and $q$. Let $x$ be a generator of $M$. Then $|[x]| \geq \phi(pq)\geq (2-1)(3-1)=2$. So there exists $y \in [x] \setminus \{x\}$. We have $[x] \subset M {{\setminus}} N$, and since $|M {\setminus} N | \geq 3$, there exists $z \in (M {{\setminus}} N) {{\setminus}} \{x, y\}$. As $x$ and $y$ are generators of $M$, the subgraph of $\mathcal{P}(G)$ induced by $M {{\setminus}} N$ contains a cycle induced by $\{x,y,z\}$.
		By analogous argument, the subgraph of $\mathcal{P}(G)$ induced by $N {{\setminus}} M$ also contains a cycle. Hence $G \setminus (M \bigtriangleup N)$ is a cyclic vertex cutset of $\mathcal{P}(G)$. Therefore $\mathcal{P}(G)$ is cyclically separable.
	\end{proof}
	
	The converse of the above lemma is not necessarily true for an arbitrary finite group. For example, consider the dihedral group, $D_{40} =
	\langle a, b \mid a^{20} = b^2 = e, ab = ba^{-1} \rangle$, of order $40$. 
	We have $\langle a^{5} \rangle \cap \langle a^4 \rangle = \{e\}$. 
	So the subgraph of $\mathcal{P}({{D}_{40}})$ induced by $\langle a^{5} \rangle^* \cup \langle a^4 \rangle^*$ is a disconnected graph with two components induced by $\langle a^{5} \rangle^*$ and $\langle a^4 \rangle^*$. 
	The orders of $\langle a^{5} \rangle$ and $\langle a^4 \rangle$ are $4$ and $5$, respectively. Then $\langle a^{5} \rangle^*$ and $\langle a^4 \rangle^*$ are cliques of size at least three in $\mathcal{P}({{D}_{40}})$, and so the subgraphs induced by them contain cycles. Thus ${D}_{40} \setminus (\langle a^{5} \rangle^* \cup \langle a^4 \rangle^*)$ is a cyclic vertex cutset of $\mathcal{P}({{D}_{40}})$. As a result, $\mathcal{P}({{D}_{40}})$ is cyclically separable. Whereas, the maximal cyclic subgroups of $D_{40}$ are $\langle a \rangle$, $\langle b \rangle$, $\langle ab \rangle$, $\dots$, $\langle a^{19}b \rangle$, where $|\langle a \rangle| = 20$ and $|\langle a^ib \rangle| = 2$ for $0 \leq i \leq 19$. Then $\text{d}(\langle a \rangle, \langle a^ib \rangle) = \min\{19,1\} = 1$ for $0 \leq i \leq 19$, and that $\text{d}(\langle a^ib \rangle, \langle a^jb \rangle) = \min\{1,1\} = 1$ for $0 \leq i, j \leq 19$, $i \neq j$. So we get $\partial(D_{40}) = 1$.

	For any finite $p$-group $G$, we now show that the conditions in (iii) of Theorem \ref{mainthm} are necessary for  $\partial(G) \geq 3$ to hold.
	
	\break
	
	\begin{lemma}\label{lemma2}
		For any finite $p$-group $G$, $\partial(G) \geq 3$ holds if $G$ satisfies one of the following conditions:
		\begin{enumerate}[\rm(a)]
			\item $p>3$ and $G$ is non-cyclic,
			\item $p=3$ and $G$ has at least two maximal cyclic subgroups of order greater than $3$,
			\item $p=2$ and $G$ has at least two maximal cyclic subgroups of order greater than $4$, or that $G$ has at least two maximal cyclic subgroups of order greater than $2$ with trivial intersection.
		\end{enumerate}
	\end{lemma}
	
	\begin{proof}
		(i) Let $p>3$ and $G$ be non-cyclic. Then $G$ has at least two maximal cyclic subgroups $M$ and $N$.
		As $|M|\geq p$, we have $|\widetilde{M}| \geq \phi(p) \geq \phi(5) = 4 $. Similarly, $|\widetilde{N}| \geq = 4 $. 
		
		\noindent
		(ii) Next let $p=3$ and let $G$ have at least two maximal cyclic subgroups $M$ and $N$ of order greater than $3$. Then $|\widetilde{M}| \geq \phi(9) = 6 $ and $|\widetilde{N}| \geq \phi(9) = 6 $. 
		
		\noindent
		(iii) Now let $p=2$ and let $G$ have at least two maximal cyclic subgroups $M$ and $N$ of order greater than $4$. Then  $|\widetilde{M}| \geq \phi(8) = 4 $ and $|\widetilde{N}| > \phi(8) = 4 $. 
		
		In each of the above cases, $|M {\setminus} N| \geq |\widetilde{M}| > 3$ and $|N {\setminus} M| \geq |\widetilde{N}| > 3$. Hence $\partial(G) \geq \text{d}(M,N) > 3$.
		
		\noindent
		(iv) Finally, let $p=2$ and let $G$ have at least two maximal cyclic subgroups $M$ and $N$ of order greater than $2$ with trivial intersection. Then $|M|\geq 4$ and $|N|\geq 4$. Thus $\partial(G) \geq \text{d}(M,N) = \min\{|M {\setminus} N|, |N {\setminus} M|\} = \min\{|{M^*}|, |{N^*}|\} \geq 3$.
	\end{proof}
	
	We now prove Theorem \ref{mainthm}.
	
	\begin{proof}[Proof of Theorem \ref{mainthm}]
		By Lemma \ref{lemma1}, we have (ii) $\implies$ (i). Whereas, by Lemma \ref{lemma2}, (iii) $\implies$ (ii). So to complete the proof,  we need to prove that (i) $\implies$ (iii). 
		
		Suppose that $\mathcal{P}(G)$ is cyclically separable. Then $\mathcal{P}(G)$ is not complete, and so $G$ is non-cyclic. 
		Since $e$ is adjacent to every other vertex of $\mathcal{P}(G)$, $e \in S$ for any cyclic vertex cutset of $\mathcal{P}(G)$.

		First we assume that $p=3$. If all maximal cyclic subgroups of $G$ are of order $3$, then 
		\begin{align}
			\mathcal{P}(G) = \mathcal{P}(e) \vee \left\{ \mathcal{P}({M_1^*}) + \mathcal{P}({M_2^*}) + \dots + \mathcal{P}({M_r^*}) \right\},
		\end{align}
		where $r \geq 2$ and $M_1, M_2, \dots M_r$ are the maximal cyclic subgroups of $G$. 
		Note that (1)  holds because $M_i \cap M_j = \{e\}$ for all $1 \leq i,j \leq r$, $i \neq j$. 
		So $\mathcal{P}(G^*) \cong r K_2$ as $\mathcal{P}({M_i^*}) \cong K_2$ for all $1 \leq i \leq r$. Thus $\mathcal{P}(G^*)$ is disconnected and contains no cycle.
		
		On the other hand, if $G$ has only one maximal cyclic subgroup of order greater than $3$, say $M$, then  
		\begin{align}
			\mathcal{P}(G) = \mathcal{P}(e) \vee \left\{ \mathcal{P}({M^*}) + \mathcal{P}({M_1^*}) + \mathcal{P}({M_2^*}) + \dots + \mathcal{P}({M_s^*}) \right\},
		\end{align}
		where $M_1, M_2, \dots , M_s$ are the maximal cyclic subgroups of order $3$ in $G$. Observe that $M \cap M_i = \{e\}$ for all $1 \leq i \leq r$, and $M_i \cap M_j = \{e\}$ for all $1 \leq i,j \leq r$, $i \neq j$. We thus have $\mathcal{P}(G^*) \cong K_{|M|-1} + s K_2$ as $\mathcal{P}({M^*}) \cong K_{|M|-1}$ and $\mathcal{P}({M_i^*}) \cong K_2$ for all $1 \leq i \leq s$. So  $\mathcal{P}(G^*)$ is disconnected and has exactly one component $\mathcal{P}({M^*})$ containing a cycle. Additionally, since $M^*$ is a clique in $\mathcal{P}(G^*)$, no futher vertex deletion will produce two components containing cycles. Thus $\mathcal{P}(G)$ has no cyclic vertex cutset.
		
		So for the case $p=3$, if $G$ has at most one maximal cyclic subgroup of order greater than $3$, then it leads to the contradiction that $\mathcal{P}(G)$ is cyclically separable. Hence $G$ has at least two maximal cyclic subgroups of order greater than $3$.
		
		Next, we assume that $p=2$. Note that for any maximal cyclic subgroup $M$ of order $2$ in $G$, $\mathcal{P}({M^*}) \cong K_1$.
		So, by an argument similar to above, if $G$ has all maximal cyclic subgroups of order $2$, then  $\mathcal{P}(G^*) \cong r K_1$ for some integer $r \geq 2$. Then $\mathcal{P}(G^*)$ is disconnected and contains no cycle. Whereas, if $G$ has only one maximal cyclic subgroup of order greater than $2$, say $N$, then $\mathcal{P}(N^*) \cong K_{|N|-1}$ and so
		$\mathcal{P}(G^*) \cong K_{|N|-1} + s K_1$ for some positive integer $s$. Thus $\mathcal{P}(G^*)$ is disconnected, and since $|N| \geq 4$, $\mathcal{P}(G^*)$ has exactly one component $\mathcal{P}(N^*)$ containing a cycle. However, $N^*$ is a clique in $\mathcal{P}(G^*)$. 
		So no further vertex deletion in $\mathcal{P}(G^*)$ will result in two components containing cycles. Hence the above two subcases of $p=2$ contradict the fact that $\mathcal{P}(G)$ is cyclically separable. Consequently, $G$ has at least two maximal cyclic subgroups of order greater than $2$.

		Now, let $M_1, M_2,\ldots,M_r$ be the maximal cyclic subgroups of $G$ of order greater than $2$. Furthermore, suppose that $|M_i| = 4$ for all $2 \leq i \leq r$, and $M_i \cap M_j \neq \{e\}$ for all $1 \leq i, j \leq r$, $i \neq j$. Then for every $1 \leq i \leq r$, $M_1 \cap M_i$ is a cyclic subgroup of order $2$. Since $M_1$ has an unique subgroup of order $2$, say $\langle x \rangle$, we get 
		$$M_1 \cap M_2 = M_1 \cap M_3 = \dots = M_1 \cap M_r = \langle x \rangle.$$
		This implies that $\langle x \rangle \subset M_i$ for all $2 \leq i \leq r$. Additionally, as $ \mathcal{P}({M_i}) $ is a complete graph for all $1 \leq i \leq r$, $e$ and $x$ are adjacent to every other vertex in $\mathcal{P}({\cup_{i=1}^r M_i})$. Thus
		\begin{align}\label{eqB}
			\mathcal{P}({\cup_{i=1}^r M_i}) = \mathcal{P}({\langle x \rangle}) \vee \left\{ \mathcal{P}({M_1 \setminus \langle x \rangle}) + \mathcal{P}({M_2 \setminus \langle x \rangle}) + \dots + \mathcal{P}({M_r \setminus \langle x \rangle}) \right\}.
		\end{align}
		In fact, (\ref{eqB}) implies that
		$$\mathcal{P}({(\cup_{i=1}^r M_i) \setminus \langle x \rangle}) = \mathcal{P}({M_1 \setminus \langle x \rangle}) + \mathcal{P}({M_2 \setminus \langle x \rangle}) + \dots + \mathcal{P}({M_r \setminus \langle x \rangle}).$$ Since $\mathcal{P}({M_1 \setminus \langle x \rangle}) \cong K_{|M_1|-2}$ and $\mathcal{P}({M_i \setminus \langle x \rangle}) \cong K_2$ for all $2 \leq i \leq r$, we have
		\begin{align}\label{eqC}
			\mathcal{P}({(\cup_{i=1}^r M_i) \setminus \langle x \rangle}) \cong K_{|M_1|-2} + (r-1) K_2.
		\end{align}
		Thus $ \mathcal{P}({(\cup_{i=1}^r M_i) \setminus \langle x \rangle}) $ is disconnected and has at most one component containing a cycle. Moreover, all its components are induced by cliques.
		
		If $G$ has no maximal cyclic subgroup of order $2$, then $G = \cup_{i=1}^r M_i$. Hence by the above arguments, $\mathcal{P}(G)$ has no cyclic vertex cutset.
		Whereas, if $N_1, N_2,\ldots,N_s$ are the maximal cyclic subgroups of order $2$ in $G$, then 
		$$\mathcal{P}(G) = \mathcal{P}(e) \vee \left[ \mathcal{P}({(\cup_{i=1}^r M_i) \setminus \{e\}})  + \left\{ \mathcal{P}({N_1^*}) + \mathcal{P}({N_2^*}) + \dots + \mathcal{P}({N_s^*}) \right\} \right].$$
		This implies $\mathcal{P}(G^*) \cong \mathcal{P}({(\cup_{i=1}^r M_i) \setminus \{e\}})  + sK_1$, and so $ \mathcal{P}(G^*) $ is disconnected. Whereas, by (\ref{eqB}), $\mathcal{P}({(\cup_{i=1}^r M_i) \setminus \{e\}})$ is connected and it is the only component of $ \mathcal{P}(G^*) $ which contains cycles. Hence to obtain a cyclic vertex cutset of $\mathcal{P}(G)$, we must disconnect $ \mathcal{P}({(\cup_{i=1}^r M_i) \setminus \{e\}}) $, and
		to do so, we must delete the vertex $x$. Thus we conclude from (\ref{eqC}) and the subsequent argument that $\mathcal{P}(G)$ has no cyclic vertex cutset. This leads to the contradiction that $\mathcal{P}(G)$ is cyclically separable.
		
		Hence $G$ has at least two maximal cyclic subgroups of order greater than $4$, or that it has at least two maximal cyclic subgroups of order greater than $2$ with trivial intersection. Thus we have completed the proof of (i) $\implies$ (iii).
	\end{proof}
\break
	
	In the following, we recall some useful lemmas about finite $p$-groups and their power graphs.

	\begin{lemma}[\cite{robinson}]\label{pgroup}
		A finite $ p $-group has exactly one subgroup of order $ p $ if and only if it is
		cyclic or a generalized quaternion group.
	\end{lemma}
	
	\begin{lemma}[\cite{moghaddamfar2014}]\label{lemmapgroup}
		For any finite $p$-group $G$, $\mathcal{P}(G^*) $ is
		connected if and only if $ G $ is either cyclic or generalized quaternion.
	\end{lemma}
	
	\begin{lemma}[\cite{panda2018a}]\label{lemmapgroup2}
		If $ G $ is a finite $ p $-group, then each component of $\mathcal{P}(G^*) $ has exactly $ p - 1 $ elements of order $ p $.
	\end{lemma}
	
	We next proceed to prove Theorem \ref{mainthm2}.
	
	\begin{proof}[Proof of Theorem \ref{mainthm2}]
		Let the finite $p$-group $G$ be such that $\kappa(\mathcal{P}(G)) = c\kappa(\mathcal{P}(G))$. Since $ \mathcal{P}(G) $ is a finite graph, $ \kappa(\mathcal{P}(G)) $ is finite. So $c\kappa(\mathcal{P}(G))$ is finite as well. Thus $\mathcal{P}(G)$ is cyclically separable. Then $\mathcal{P}(G)$ is not complete, and so $G$ is non-cyclic.
		
		Consider the case $p \in \{2,3\}$. Then by Theorem \ref{mainthm}, one of the following holds:
		\begin{enumerate}[\rm(a)]
			\item $p=3$ and $G$ has at least two maximal cyclic subgroups of order greater than $3$,
			\item $p=2$ and $G$ has at least two maximal cyclic subgroups of order greater than $4$, or that $G$ has at least two maximal cyclic subgroups of order greater than $2$ with trivial intersection.
		\end{enumerate}
		
		If possible, let $G$ be generalized quaternion. Then $G$ is a $2$-group with the presentation
		$$G = \langle a,b : a^{2n} = e, a^n = b^2, ab = ba^{-1} \rangle,$$
		where $n$ is some power of $2$. Note that $\langle a \rangle$ is a maximal cyclic subgroup of $G$ of order $2n$. Whereas, for $0 \leq k \leq n-1$,  $\langle a^kb \rangle = \{e, a^kb, a^n, a^{n+k}b\}$ is a maximal cyclic subgroup of $G$ of order $4$. In fact, these are the only maximal cyclic subgroup of $G$. Then $G$ has at most one maximal cyclic subgroup of order greater than $4$. Moreover, the intersection of every pair of maximal cyclic subgroups of $G$ is $ \{e, a^n\} $, which is non-trivial. As these lead to contradiction, $G$ is not generalized quaternion. 
		
		Now that $G$ is neither cyclic nor generalized quaternion, by Lemma \ref{lemmapgroup}, $\mathcal{P}(G^*) $  is disconnected. From this and our initial assumption,  $\kappa(\mathcal{P}(G)) = c\kappa(\mathcal{P}(G)) = 1$.
		Thus $\{e\}$ is a cyclic vertex cutset of  $\mathcal{P}(G)$. 
		So at least two components of $\mathcal{P}(G^*) $ contain cycles. Since $p \in \{2,3\}$, this implies that at least two maximal cyclic subgroups of $G$ are of order greater than $p$. For $r \geq 2$, let $M_1,M_2,\dots,M_r$ be the maximal cyclic subgroups of $G$ of order greater than $p$.  If possible, suppose that $M_i \cap M_j$ is non-trivial for all $1 \leq i, j \leq r$, $i \neq j$.

		Let  $H_1$ be the subgroup of order $p$ in $M_1$.
		For $2 \leq j \leq r$, as $M_1 \cap M_j$ is non-trivial $p$-group, it has a subgroup $H_j$ of order $p$. Then $H_j$ is a subgroup of order $p$ in $M_1$ for all $2 \leq j \leq r$. However, since $M_1$ is a cyclic group, it contains an unique subgroup of order $p$. This implies that $H_1 = H_j$ for all $2 \leq j \leq r$. Then $H_1 \subseteq M_1 \cap M_j \subseteq M_j$ for all $2 \leq j \leq r$. Thus we conclude that $H_1 \subseteq \cap_{i=1}^r M_i$. Let $x \in H_1^*$. Let $y \in M_i\setminus \{x\}$ for some $1 \leq i \leq r$. Then  there exists a positive integer $m$ such that $y^{p^m}$ is of order $p$. As $x$ is of order $p$ as well, $x = y^{lp^m}$ for some integer $1\leq l < p$ . Then $y$ is adjacent to $x$. Since this holds for every $1 \leq i \leq r$, $\mathcal{P}({\cup_{i=1}^r M_i^*})$ is connected. However, we have shown that $\mathcal{P}(G^*) $ is disconnected. Thus $G \neq \cup_{i=1}^r M_i$. So $G$ has at least one maximal cyclic subgroup of order $p$. Let $N_1,N_2,\dots, N_s$ be the maximal cyclic subgroups of $G$ of order $p$. Then $N_i \cap N_j$ is trivial for all $1 \leq i, j \leq r$, $i \neq j$, and so
		$$\mathcal{P}(G^*) = \mathcal{P}({\cup_{i=1}^r M_i^*}) + \mathcal{P}({N_1^*}) + \mathcal{P}({N_2^*}) + \dots + \mathcal{P}({N_s^*}).$$
		For $1 \leq i \leq r$, as $ \mathcal{P}({N_i^*}) \cong K_{1}$ or $ \mathcal{P}({N_i^*}) \cong K_{2}$, $ \mathcal{P}({\cup_{i=1}^r M_i^*}) $ is the only component of $\mathcal{P}(G^*)$ that contains cycles. This leads to a contradiction. Therefore, $\mathcal{P}(G)$ has at least two maximal cyclic subgroups of order greater than $p$ with trivial intersection.
		
		To prove the converse, we first suppose that $p > 3$ and $G$ is a non-cyclic $p$-group. Then $G$ is neither cyclic nor generalized quaternion. So by Lemma \ref{pgroup}, $G$ has at least two  subgroups $K_1$ and $K_2$ of order $p$. Then $G$ has two maximal cyclic subgroups $M_1$ and $M_2$ such that $K_1 \subseteq M_1$ and $K_2 \subseteq M_2$. If $M_1 \cap M_2$ is non-trivial, then it contains a subgroup $K$ of order $p$. Thus $K = K_1 = K_2$, which is a contradiction. So $M_1 \cap M_2$ is trivial. 
		
		Next, suppose $p \in \{2,3\}$ and that $G$ has at least two maximal cyclic subgroups of order greater than $p$ with trivial intersection. By combining this with the case $p > 3$, $G$ has at least two maximal cyclic subgroups $M_1$ and $M_2$ such that $M_1 \cap M_2$ is trivial, and $|M_1| \geq 4$ and $|M_2|\geq 4$. Moreover, both $M_1$ and $M_2$ have $p-1$ elements of order $p$. Hence, we conclude from Lemma \ref{lemmapgroup2}, $M_1^*$ and $M_2^*$ are in different components of $\mathcal{P}(G^*)$. In particular, $\mathcal{P}(G^*)$ is disconnected. Also $ \mathcal{P}({M_1^*}) \cong K_{|M_1|-1}$ and $ \mathcal{P}({M_2^*}) \cong K_{|M_2|-1}$. Hence $\{e\}$ is a cyclic vertex cutset of $\mathcal{P}(G)$. Therefore, $\kappa(\mathcal{P}(G)) = c\kappa(\mathcal{P}(G)) = 1$.
	\end{proof}

\end{document}